\documentclass{amsart}

\usepackage{amssymb}
\usepackage{graphicx}
\usepackage{amsmath}
\usepackage{amsthm}
\usepackage{tikz}
\usetikzlibrary{arrows}

\theoremstyle{plain}
\newtheorem{theorem}{Theorem}[section]
\newtheorem{corollary}[theorem]{Corollary}
\newtheorem{lemma}[theorem]{Lemma}
\newtheorem{proposition}[theorem]{Proposition}

\theoremstyle{definition}
\newtheorem{definition}[theorem]{Definition}
\newtheorem{example}[theorem]{Example}
\newtheorem{remark}[theorem]{Remark}

\usepackage[hidelinks]{hyperref}
\hypersetup{colorlinks=true,citecolor=blue,linkcolor=blue}

\begin{document}

\title{Quasicomplemented distributive nearlattices}
\address{Facultad de Ciencias y Artes, Universidad Cat\'olica de \'Avila, C/Canteros s/n, 05005 \'Avila, Espa\~{n}a.}
\email{ismaelm.calomino@ucavila.es}
\author{Ismael Calomino}
%\address{CIC and NUCOMPA, Facultad de Ciencias Exactas, Universidad Nacional del Centro, Pinto 399, 7000 Tandil, Argentina.}
%\email{ismaelcalomino@gmail.com}
%\author{Ismael Calomino}

%\thanks{This work was supported by Consejo Nacional de Investigaciones Cient\'{i}ficas y T\'{e}cnicas (PIP 112-202001-01301, CONICET-Argentina) and Agencia Nacional de Promoci\'{o}n Cient\'{i}fica y Tecnol\'{o}gica (PICT2019-2019-00882, ANPCyT-Argentina). This project has also received funding from MOSAIC Project 101007627 (European Union's Horizon 2020 research and innovation programme under the Marie Sklodowska-Curie).}

\begin{abstract}
The aim of this paper is to study the class of quasicomplemented distributive nearlattices. We investigate $\alpha$-filters and $\alpha$-ideals in quasicomplemented distributive nearlattices and some results on ideals-congruence-kernels. Finally, we also study the notion of Stone distributive nearlattice and give a characterization by means $\sigma$-filters.
\end{abstract}
\maketitle

\section{Introduction} 

An interesting characterization of the distributivity of a lattice was given by Mandelker in \cite{Ma70}. He note that a lattice is distributive if and only if each relative annihilator is an ideal. This result was extended to different ordered algebraic structures. Varlet in \cite{Va73} proved that a semilatttice is distributive, in the sense of Gr\"{a}tzer \cite{Gra98}, if and only if each relative annihilator is an order-ideal, and Chajda and Kola\v{r}\'{\i}k in \cite{ChaKo07} prove a similar result but in distributive nearlattices. Recall that a distributive nearlattice is a join-semilattice with a greatest element where each principal filter is a bounded distributive lattice. It follows that distributive nearlattices are a natural generalization of semi-boolean algebras \cite{Abb67} and bounded distributive lattices. Concerning the structure and properties of distributive nearlattices, the reader is referred to \cite{Hi80,ChaHaKu07,ChaKo08,ArKin11,CaCe15,CaCe18,Go18,Ca19,CaGo21,Go22}. In particular, in \cite{CaCe15}, the authors presented an alternative definition of relative annihilator in distributive nearlattices different from that given in \cite{ChaKo07} and established new equivalences of the distributivity. Later, using the results developed in \cite{CaCe15}, the notion of annihilator was studied in a series of papers \cite{CaCe18,Ca19,CaGo21}.

The class of quasicomplemented lattices was introduced in \cite{Va68} as a generalization of pseudocomplemented distributive lattices and was studied by Cornish in \cite{Cor73,Cor74} together with the notions of annihilator and $\alpha$-ideal. These results were extended to the class of quasicomplemented distributive semilattices in \cite{CaCe21} and recently to the variety of residuated lattices by Rasouli in \cite{Ra20,Ra21}. Motivated by the previous works, the main aim of this paper is to study the class of quasicomplemented distributive nearlattices. We also introduce Stone distributive nearlattices and obtain a characterization in terms of certain particular filters.

This paper is organized as follows. In Section \ref{Sec2} we recall some definitions and results about distributive nearlattices. To be more precise, we recall the notions of annihilator, normal distributive nearlattice, $\alpha$-filter and $\alpha$-ideal which we are going to need throughout the paper. In Section \ref{Sec3} we study the class of quasicomplemented distributive nearlattices. We give characterizations through $\alpha$-filters and $\alpha$-ideals, and also some results about ideals-congruence-kernels. Finally, in Section \ref{Sec4}, Stone distributive nearlattices are study. We investigate the relationship with quasicomplemented and normal distributive nearlattices and we end with a charaterization of Stone distributive nearlattices in terms of $\sigma$-filters.

\section{Preliminaries} \label{Sec2}

Let ${\bf{A}} = \langle A, \vee, 1 \rangle$ be a join-semilattice with greatest element. A {\it{filter}} is a subset $F$ of $A$ such that $1 \in F$, if $a \leq b$ and $a \in F$, then $b \in F$ and if $a,b \in F$, then $a \wedge b \in F$, whenever $a \wedge b$ exists. Denote by ${\rm{Fi}}(A)$ the set of all filters of $A$. If $X$ is a non-empty subset of $A$, the smallest filter containing $X$ is called the {\it{filter generated by $X$}} and will be denoted by ${\rm{Fig}}(X)$. A filter $G$ is said to be finitely generated if $G={\rm{Fig}}(X)$ for some finite non-empty subset $X$ of $A$. If $X=\{a\}$, then ${\rm{Fig}}(\{a\}) = [a) = \{ x \in A \colon a \leq x \}$ called the {\it{principal filter of $a$}}. A subset $I$ of $A$ is called an {\it{ideal}} of $A$ if $a \leq b$ and $b \in I$, then $a \in I$ and if $a,b \in I$, then $a \vee b \in I$. If $X$ is a non-empty set of $A$, the smallest ideal containing $X$ is called the {\it{ideal generated by $X$}} and will be denoted by ${\rm{Idg}}(X)$. Then ${\rm{Idg}}(X) = \{ a \in A \colon \exists x_{1}, \dots, x_{n} \in X (a \leq x_{1} \vee \dots \vee x_{n}) \}$. If $X=\{a\}$, then ${\rm{Idg}}(\{a\}) = (a] = \{ x \in A \colon x \leq a \}$ called the {\it{principal ideal of $a$}}. A non-empty proper ideal $P$ is {\it{prime}} if for every $a,b \in A$, $a \wedge b \in P$ implies $a \in P$ or $b \in P$, whenever $a \wedge b$ exists. Denote by ${\rm{Id}}(A)$ and ${\rm{X}}(A)$ the set of all ideals and prime ideals of $A$, respectively. Finally, a non-empty ideal $I$ of $A$ is {\it{maximal}} if it is proper and for every $J \in {\rm{Id}}(A)$, if $I \subseteq J$, then $J=I$ or $J=A$. Denote by ${\rm{X}}_{m}(A)$ the set of all maximal ideals of $A$. It follows that ${\rm{X}}_{m}(A) \subseteq {\rm{X}}(A)$.

\subsection{Distributive nearlattices}

We recall the basics results about distributive nearlattices, where our main reference is \cite{ChaHaKu07,ArKin11,CaCe15,Ca19,CaGo21}.

\begin{definition}
Let ${\bf{A}}$ be a join-semilattice with greatest element. We say that ${\bf{A}}$ is a {\it{distributive nearlattice}} if each principal filter is a bounded distributive lattice.
\end{definition}

Distributive nearlattices are in a one-to-one correspondence with certain ternary algebras \cite{Hi80,ChaKo08}. Then a smaller equational base for the ternary algebras is given by Ara\'{u}jo and Kinyon in \cite{ArKin11}. 

\begin{theorem} \cite{ArKin11} \label{ternary operation}
Let ${\bf{A}} = \langle A,\vee, 1 \rangle$ be a distributive nearlattice. Let $m \colon A^{3} \to A$ be the ternary operation given by $m(x,y,z) = (x \vee z) \wedge_{z} (y \vee z)$, where $\wedge_{z}$ denotes the meet in $[z)$. Then the structure ${\bf{A}}_{*} = \langle A, m, 1 \rangle$ satisfies the following identities:
\begin{enumerate}
\item $m(x,x,1) = 1$,
\item $m(x,y,x) = x$,
\item $m(m(x,y,z),m(y,m(u,x,z),z),w) = m(w,w,m(y,m(x,u,z),z))$,
\item $m(x,m(y,y,z),w) = m(m(x,y,w),m(x,y,w),m(x,z,w))$.
\end{enumerate}
Conversely, let ${\bf{A}} = \langle A, m, 1 \rangle$ be an algebra of type $(3,0)$ satisfying the identities $(1) - (4)$. If we define the binary operation $x \vee y = m(x,x,y)$, then ${\bf{A}}^{*} = \langle A, \vee, 1 \rangle$ is a distributive nearlattice. Moreover, $({\bf{A}}_{*})^{*} = {\bf{A}}$ and $({\bf{A}}^{*})_{*} = {\bf{A}}$.
\end{theorem}

By Theorem \ref{ternary operation} the distributive nearlattices can be presented in two equivalent forms and both are useful to different purposes. Examples of distributive nearlattices are the bounded distributive lattices and semi-boolean algebras \cite{Abb67}.

Let ${\bf{A}}$ be a distributive nearlattice. Recall that we have the following characterization of the filter generated by a subset $X \subseteq A$:
\[
{\rm{Fig}}(X) = \{ a \in A \colon \exists x_{1}, \dots, x_{n} \in [X) (a = x_{1} \wedge \dots \wedge x_{n}) \}.
\]
We consider ${\rm{Fi}}({\bf{A}}) = \langle {\rm{Fi}}(A), \veebar, \cap, \to, \{1\}, A \rangle$, where the least element is $\{1\}$, the greatest element is $A$, and for each $G,H \in {\rm{Fi}}(A)$, we have $G \veebar H = {\rm{Fig}}(G \cup H)$ and $G \to H = \{ x \in A \colon [x) \cap G \subseteq H \}$. Then ${\rm{Fi}}({\bf{A}})$ is a Heyting algebra \cite{CaCe15}. In particular, it is easy to see that the pseudocomplement of a filter $F$ of $A$ is given by $F^{\top} = F \to \{1\} = \{ x \in A \colon \forall f \in F (x \vee f = 1) \}$.

\begin{theorem} \cite{ChaHaKu07} \label{prime_theo}
Let ${\bf{A}}$ be a distributive nearlattice. Let $I \in {\rm{Id}}(A)$ and $F \in {\rm{Fi}}(A)$ be such that $I \cap F = \emptyset$. Then there exists $P \in {\rm{X}}(A)$ such that $I \subseteq P$ and $P \cap F = \emptyset$.
\end{theorem}

\subsection{Annihilators and normal distributive nearlattices}

We introduce the relative annihilators in distributive nearlattices. The following definition was given in \cite{CaCe15} as an alternative from that given in \cite{ChaKo07}.

\begin{definition}
Let ${\bf{A}}$ be a join-semilattice with greatest element and $a,b \in A$. The {\it{annihilator of $a$ relative to $b$}} is the set 
\[
a \circ b = \{ x \in A \colon b \leq x \vee a \}.
\]
In particular, $a^{\top} = a \circ 1 = \{ x \in A \colon x \vee a = 1 \}$ is called the {\it{annihilator of $a$}}.
\end{definition}

Due to the results developed in \cite{CaCe15}, if ${\bf{A}}$ is a distributive nearlattice we have $a \circ b \in {\rm{Fi}}(A)$, for all $a,b \in A$. Let $a \in A$ and consider  
\[
a^{\top \top} = (a^{\top})^{\top} = \{ y \in A \colon \forall x \in a^{\top} (y \vee x = 1)\}. 
\]
It follows $a^{\top\top} \in {\rm{Fi}}(A)$. We say that an element $a \in A$ is {\it{dense}} if $a^{\top} = \{1\}$. Denote by ${\rm{D}}(A)$ the set of all dense elements of $A$. Also, we say that a filter $F$ of $A$ is {\it{dense}} if it is dense in the Heyting algebra ${\rm{Fi}}(A)$, i.e., $F^{\top} = \{1\}$. 

\begin{lemma} \cite{Ca19,CaGo21} \label{prop_anni}
Let ${\bf{A}}$ be a distributive nearlattice. Let $a,b \in A$ and $I \in {\rm{Id}}(A)$. Then:
\begin{enumerate}
\item $[a) \subseteq a^{\top \top}$.
\item $a^{\top \top \top} = a^{\top}$.
\item If $a \leq b$, then $a^{\top} \subseteq b^{\top}$.
\item $a^{\top} \subseteq b^{\top}$ if and only if $b^{\top \top} \subseteq a^{\top \top}$.
\item $(a \wedge b)^{\top} = a^{\top} \cap b^{\top}$, whenever $a \wedge b$ exists.
\item $(a \vee b)^{\top \top} = a^{\top \top} \cap b^{\top \top}$.
\item $I \cap a^{\top} = \emptyset$ if and only if there exists $U \in {\rm{X}}_{m}(A)$ such that $I \subseteq U$ and $a \in U$.
\item If $U \in {\rm{Id}}(A)$, then $U \in {\rm{X}}_{m}(A)$ if and only if $\forall a \in A (a \notin U \Leftrightarrow U \cap a^{\top} \neq \emptyset)$.
\item If $U \in {\rm{X}}_{m}(A)$, then $\forall a \in A (a \notin U \Leftrightarrow U \cap a^{\top \top} = \emptyset)$.
\end{enumerate}
\end{lemma}

We are interested in a particular class of distributive nearlattices presented in \cite{CaCe15}, which are a generalization of normal lattices given by Cornish in \cite{Cor72}.

\begin{definition} \label{def_normal}
Let ${\bf{A}}$ be a distributive nearlattice. We say that ${\bf{A}}$ is {\it{normal}} if each prime ideal is contained in a unique maximal ideal.
\end{definition}

\begin{theorem} \cite{CaCe15} \label{normal}
Let ${\bf{A}}$ be a distributive nearlattice. Then the following conditions are equivalent:
\begin{enumerate}
\item ${\bf{A}}$ is normal.
\item $(a \vee b)^{\top} = a^{\top} \veebar b^{\top}$, for all $a,b \in A$.
\item If $P \in {\rm{X}}(A)$ and $a,b \in A$ such that $a \vee b =1$, then $P \cap a^{\top} \neq \emptyset$ or $P \cap b^{\top} \neq \emptyset$.
\end{enumerate}
\end{theorem}

Let ${\bf{A}}$ be a normal distributive nearlattice and ${\rm{R}}(A) = \{ a^{\top} \colon a \in A \}$. We define $\overline{m} \colon {\rm{R}}(A)^{3} \to {\rm{R}}(A)$ by 
\[
\overline{m} (a^{\top}, b^{\top}, c^{\top}) = (a^{\top} \veebar c^{\top}) \cap ( b^{\top} \veebar c^{\top}).
\]
By Lemma \ref{prop_anni} and Theorem \ref{normal} we have that the structure ${\rm{R}}({\bf{A}}) = \langle {\rm{R}}(A), \overline{m}, A \rangle$ is a distributive nearlattice (for more details see \cite{Ca19,CaGo21}).

\subsection{$\alpha$-filters, $\alpha$-ideals and congruences}

Now we consider some particular classes of filters and ideals studied in \cite{Ca19,CaGo21}, and also recall the concept of congruence in distributive nearlattices.

\begin{definition} \cite{Ca19,CaGo21}
Let ${\bf{A}}$ be a distributive nearlattice. Then:
\begin{itemize}
\item We say that a filter $F$ of $A$ is an $\alpha$-filter if $a^{\top \top} \subseteq F$, for all $a \in F$.
\item We say that an ideal $I$ of $A$ is an $\alpha$-ideal if for each $a \in A$, $I \cap a^{\top \top} \neq \emptyset$ implies $a \in I$.
\end{itemize}
\end{definition}

Let ${\bf{A}}$ be a distributive nearlattice. Denote by ${\rm{Fi}}_{\alpha}(A)$ and ${\rm{Id}}_{\alpha}(A)$ the set of all $\alpha$-filters and $\alpha$-ideals of $A$, respectively. Denote by ${\rm{X}}_{\alpha}(A)$ the set of all prime $\alpha$-ideals of $A$. It is easy to see that for each $a \in A$, $a^{\top}$ and $a^{\top \top}$ are $\alpha$-filters of $A$. On the other hand, ${\rm{D}}(A)$ is an example of $\alpha$-ideal of $A$. Moreover, ${\rm{D}}(A)$ is the smallest $\alpha$-ideal of $A$. Note that by Lemma \ref{prop_anni} every maximal ideal is in particular an $\alpha$-ideal.

\begin{proposition} \cite{CaGo21} \label{theocaract1}
Let ${\bf{A}}$ be a normal distributive nearlattice and $I \in {\rm{Id}}(A)$. Then 
\[
{\rm{Idg}}_{\alpha}(I) = \{ x \in A \colon \exists i \in I (x^{\top} \subseteq i^{\top}) \}
\]
is the smallest $\alpha$-ideal containing $I$. In particular, for each $a \in A$ we have ${\rm{Idg}}_{\alpha}((a]) = (a]_{\alpha} = \{ x \in A \colon x^{\top} \subseteq a^{\top} \}$.
\end{proposition}

\begin{proposition} \cite{Ca19} \label{theocaract1}
Let ${\bf{A}}$ be a distributive nearlattice and $F \in {\rm{Fi}}(A)$. Then the following conditions are equivalent:
\begin{enumerate}
\item $F$ is an $\alpha$-filter.
\item If $a^{\top}=b^{\top}$ and $a\in F$, then $b\in F$.
\item $F= \bigcup \{ a^{\top \top} \colon a \in F \}$.
\end{enumerate} 
Moreover, if ${\bf{A}}$ is normal, then 
\[
\alpha(F) = \{ x \in A \colon \exists a \in F (a^{\top} \subseteq x^{\top}) \}
\]
is the smallest $\alpha$-filter containing $F$.
\end{proposition}

Then ${\rm{Fi}}_{\alpha}({\bf{A}})= \langle {\rm{Fi}}_{\alpha}(A), \underline{\sqcup}, \cap, \Rightarrow, \{1\}, A \rangle$ is a Heyting algebra, where for each $F,G \in {\rm{Fi}}_{\alpha}(A)$, we have $F {\underline{\sqcup}} G = \alpha(F {\veebar} G)$ and $F \Rightarrow G = \alpha(F \to G)$ \cite{Ca19}.

\begin{theorem} \cite{CaGo21} \label{prime alpha_theo}
Let ${\bf{A}}$ be a distributive nearlattice. Let $I \in {\rm{Id}}(A)$ and $F \in {\rm{Fi}}_{\alpha}(A)$ be such that $I \cap F = \emptyset$. Then there exists $P \in {\rm{X}}_{\alpha}(A)$ such that $I \subseteq P$ and $P \cap F = \emptyset$.
\end{theorem}

A {\it{congruence}} of a distributive nearlattice ${\bf{A}}$ is an equivalence relation $\theta \subseteq A \times A$ such that: $(i)$ if $(a,b), (c,d) \in \theta$, then $(a \vee c, b \vee d) \in \theta$, and $(ii)$ if $(a,b), (c,d) \in \theta$ and $a \wedge c$, $b \wedge d$ exist, then $(a \wedge c, b \wedge d) \in \theta$ \cite{ChaHaKu07}. Then an equivalence relation $\theta$ is a congruence of ${\bf{A}}$ if and only if $\theta$ is compatible with the ternary operation $m$. If ${\bf{A}}$ is normal, then $\Theta^{\top} \subseteq A \times A$ defined by $(a,b) \in \Theta^{\top}$ if and only if $a^{\top} = b^{\top}$ is a congruence of ${\bf{A}}$ such that ${\bf{A}}/ \Theta^{\top}$ is isomorphic to ${\rm{R}}({\bf{A}})$ via the map $\rho \colon A \to {\rm{R}}(A)$ given by $\rho(x) = x^{\top}$. So, $m(x,y,z)^{\top} = \overline{m} (a^{\top}, b^{\top}, c^{\top})$ and $\rho$ is an onto homomorphism between the distributive nearlattices ${\bf{A}}$ and {\rm{R}}({\bf{A}}) such that $\Theta^{\top} = {\rm{Ker}}(\rho)$ \cite{Ca19,CaGo21}. The next result gives a necessary and sufficient conditions via $\alpha$-filters so that $\rho$ is 1-1. First we recall that a filter $F$ of $A$ is {\it{prime}} if for every $a,b \in A$, $a \vee b \in F$ implies $a \in F$ or $b \in F$.

\begin{theorem} \label{theo_ro}
Let ${\bf{A}}$ be a distributive nearlattice. Then the following conditions are equivalent:
\begin{enumerate}
\item $\rho$ is 1-1.
\item Every filter is an $\alpha$-filter.
\item Every prime filter is an $\alpha$-filter.
\end{enumerate}
\end{theorem}

\begin{proof}
$(1) \Rightarrow (2)$ It follows by Proposition \ref{theocaract1} and $(2) \Rightarrow (3)$ it is immediate. We only prove $(3) \Rightarrow (1)$. Let $a,b \in A$ such that $a^{\top} = b^{\top}$ and $a \neq b$. Without loss of generality, suppose $a \nleq b$. Let $\mathcal{F} = \{ F \in {\rm{Fi}}(A) \colon b \in F, \hspace{0.1cm} a \notin F  \}$. Then $[a) \in \mathcal{F}$ and $\mathcal{F} \neq \emptyset$. By Zorn's Lemma, there is a maximal element $G$ of $\mathcal{F}$. We see that $G$ is prime. Let $x,y \in A$ be such that $x \vee y \in G$ and $x,y \notin G$. Consider the filters $G_{x} = {\rm{Fig}}(G \cup \{x\})$ and $G_{y} = {\rm{Fig}}(G \cup \{y\})$. Thus, $G_{x}, G_{y} \notin \mathcal{F}$ and $a \in G_{x} \cap G_{y}$. Since ${\rm{Fi}}({\bf{A}})$ is a distributive lattice we have 
\[
G_{x} \cap G_{y} = (G \veebar [x)) \cap (G \veebar [y)) = G \veebar [x \vee y) = G
\]
and $a \in G$, which is a contradiction. Then $G$ is a prime filter. So, by hypothesis, $G$ is an $\alpha$-filter. Then $a \in a^{\top \top} = b^{\top \top} \subseteq G$ and $a \in G$ which is impossible and we conclude $a \leq b$. Analogously we can prove $b \leq a$ and therefore $\rho$ is 1-1.
\end{proof}

\section{Quasicomplemented distributive nearlattices} \label{Sec3}

In this section, the main of this paper, we study the notions of $\alpha$-ideal, $\alpha$-filter and ideal-congruence-kernel in quasicomplemented distributive nearlattice.

\begin{definition} \label{def_quasi}
Let ${\bf{A}}$ be a distributive nearlattice. We say that ${\bf{A}}$ is {\it{quasicomplemented}} if for each $a \in A$, there exists $b \in A$ such that $a^{\top \top} = b^{\top}$.
\end{definition}

\begin{remark} \label{cuasi-dense}
If ${\bf{A}}$ is a quasicomplemented distributive nearlattice, then ${\bf{A}}$ has a dense element. Indeed, as $1 \in A$, there exists $d \in A$ such that $1^{\top \top} = d^{\top}$. Since $1^{\top \top} = \{1\}$, we have $d^{\top} = \{1\}$ and ${\rm{D}}(A)$ is non-empty. Therefore, every $\alpha$-ideal of $A$ is non-empty.
\end{remark}

\begin{lemma}
Let ${\bf{A}}$ be a distributive nearlattice. If each annihilator is a principal filter, then ${\bf{A}}$ is quasicomplemented.
\end{lemma}

\begin{proof}
Let $a \in A$ and consider $a^{\top}$. Then there is $b \in A$ such that $a^{\top} = [b)$. So, $a^{\top \top} = [b)^{\top} = {\rm{Fig}}(\{b\})^{\top} = b^{\top}$, i.e., $a^{\top \top} = b^{\top}$ and ${\bf{A}}$ is quasicomplemented.
\end{proof}

We have the following characterizations.

\begin{theorem} \label{theo_equi1}
Let ${\bf{A}}$ be a distributive nearlattice. Then the following conditions are equivalent:
\begin{enumerate}
\item ${\bf{A}}$ is quasicomplemented.
\item For each $a \in A$, there exists $b \in A$ such that $a^{\top} \cap b^{\top} = \{1\}$ and $a \vee b = 1$.
\end{enumerate}
\end{theorem}

\begin{proof}
$(1) \Rightarrow (2)$ It follows from Definition \ref{def_quasi}.

$(2) \Rightarrow (1)$ Let $a \in A$. Then there exists $b \in A$ such that $a^{\top} \cap b^{\top} = \{1\}$ and $a \vee b = 1$. As ${\rm{Fi}}({\bf{A}})$ is a pseudocomplemented distributive lattice, then $b^{\top} \subseteq a^{\top \top}$. For the other inclusion, if $x \in a^{\top \top}$, then since $b \in a^{\top}$ we have $x \vee b = 1$ and $x \in b^{\top}$. Hence, $a^{\top \top} = b^{\top}$.
\end{proof}

Let ${\bf{A}}$ be a distributive nearlattice. For a subset $Y \subseteq A$ we define  
\[
\psi[Y] = \{ P \in {\rm{X}}(A) \colon Y \cap P \neq \emptyset \}.
\]

\begin{theorem}
Let ${\bf{A}}$ be a distributive nearlattice. Then the following conditions are equivalent:
\begin{enumerate}
\item ${\bf{A}}$ is quasicomplemented.
\item For each $a \in A$, there exists $b \in A$ such that 
\[
\psi [a^{\top \top}] \cap {\rm{X}}_{m}(A) = \psi[b^{\top}] \cap {\rm{X}}_{m}(A).
\]
\end{enumerate}
\end{theorem}

\begin{proof}
$(1) \Rightarrow (2)$ It is immediate by Definition \ref{def_quasi}.

$(2) \Rightarrow (1)$ Let $a \in A$. Then there exists $b \in A$ such that $\psi [a^{\top \top}] \cap {\rm{X}}_{m}(A) = \psi[b^{\top}] \cap {\rm{X}}_{m}(A)$. We prove $a^{\top \top} = b^{\top}$. Let $x \in a^{\top \top}$ and suppose $x \notin b^{\top}$. Then by Lemma \ref{prop_anni} there is $U \in {\rm{X}}_{m}(A)$ such that $x \in U$ and $b \in U$. So, $x \in a^{\top \top} \cap U$ and $U \in \psi [a^{\top \top}] \cap {\rm{X}}_{m}(A)$. Then $U \in \psi[b^{\top}] \cap {\rm{X}}_{m}(A)$ and $b^{\top} \cap U \neq \emptyset$. Thus, there is $y \in b^{\top}$ such that $y \in U$. As $b \in U$ and $y \vee b = 1$ we have $U = A$, which is a contradiction because $U$ is maximal. Therefore, $a^{\top \top} \subseteq b^{\top}$. The other inclusion is similar and ${\bf{A}}$ is quasicomplemented.
\end{proof}

\begin{corollary} \label{normal-quasi-R(A)}
Let ${\bf{A}}$ be a normal and quasicomplemented distributive nearlattice. Then every element of ${\rm{R}}(A)$ has its complement in ${\rm{R}}(A)$ thus ${\rm{R}}({\bf{A}})$ is a Boolean algebra.
\end{corollary}

\begin{proof}
It is a consequence of Remark \ref{cuasi-dense}, Theorems \ref{theo_equi1} and \ref{normal}.
\end{proof}

The following two examples show that there is not necessarily a relationship between the class of normal distributive nearlattices and the class of quasicomplemented distributive nearlattices.

\begin{example} \label{normal no cuasi}
Let ${\bf{A}}$ be the distributive nearlattice given by
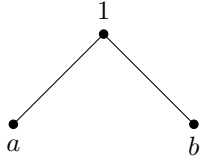
\begin{center}
\begin{tikzpicture}[scale=.8,mipunto/.style ={color=black}]
\filldraw[mipunto] (0, 0.4) node{$1$};
\filldraw[mipunto] (-1.5,-1.85) node{$a$};
\filldraw[mipunto] (1.5,-1.85) node{$b$};

\filldraw[mipunto] (-1.5,-1.5) circle (2pt);
\filldraw[mipunto] (0,0) circle (2pt); 
\filldraw[mipunto] (1.5,-1.5) circle (2pt); 
\draw (-1.5,-1.5) -- (0,0) -- (1.5,-1.5);
\end{tikzpicture}
\end{center}
Then ${\bf{A}}$ is normal but not quasicomplemented because there does not exist $b \in A$ such that $1^{\top \top} = b^{\top}$.
\end{example}

\begin{example} \label{cuasi no normal}
Consider the distributive nearlattice ${\bf{A}}$ given by 
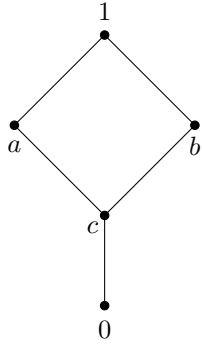
\begin{center}
\begin{tikzpicture}[scale=.8,mipunto/.style ={color=black}]
\filldraw[mipunto] (0, 0.4) node{$1$};
\filldraw[mipunto] (-1.5,-1.85) node{$a$};
\filldraw[mipunto] (1.5,-1.85) node{$b$};
\filldraw[mipunto] (-0.2, -3.2) node{$c$};
\filldraw[mipunto] (0, -4.9) node{$0$};

\filldraw[mipunto] (-1.5,-1.5) circle (2pt);
\filldraw[mipunto] (0,0) circle (2pt); 
\filldraw[mipunto] (1.5,-1.5) circle (2pt); 
\filldraw[mipunto] (0,-3) circle (2pt); 
\filldraw[mipunto] (0,-4.5) circle (2pt); 

\draw (-1.5,-1.5) -- (0,0) -- (1.5,-1.5) -- (0,-3) -- (0,-4.5);
\draw (-1.5,-1.5) -- (0,-3);
\end{tikzpicture}
\end{center}
Then ${\bf{A}}$ is quasicomplemented however it is not normal because $(a \vee b)^{\top} = A$ and $a^{\top} \veebar b^{\top} = {\rm{Fig}}(a^{\top} \cup b^{\top}) = \{1,a,b,c\}$, i.e., $(a \vee b)^{\top} \neq a^{\top} \veebar b^{\top}$.
\end{example}

By Lemma \ref{prop_anni} every maximal ideal is a prime $\alpha$-ideal. In the class of normal and quasicomplemented distributive nearlattices the converse holds also.

\begin{theorem} \label{prime alpha implica maximal}
Let ${\bf{A}}$ be a normal and quasicomplemented distributive nearlattice. Then every prime $\alpha$-ideal is maximal.
\end{theorem}

\begin{proof}
Let $P \in {\rm{X}}_{\alpha}(A)$ and $a \in A$ such that $a \notin P$. Since ${\bf{A}}$ is quasicomplemented, there exists $b \in A$ such that $a^{\top \top} = b^{\top}$. So, $b \in b^{\top \top} = a^{\top}$. Prove $(a]_{\alpha} \cap (b]_{\alpha} \subseteq {\rm{D}}(A)$. If $x \in (a]_{\alpha} \cap (b]_{\alpha}$, then $x^{\top} \subseteq a^{\top}$ and $x^{\top} \subseteq b^{\top}$. Let $y \in x^{\top}$. Then $y \in a^{\top} = b^{\top \top}$ and $y \in b^{\top}$ implies $y=1$. Thus, $x^{\top} = \{1\}$ and $x \in {\rm{D}}(A)$. As ${\rm{D}}(A)$ is the smallest $\alpha$-ideal of $A$ we have $(a]_{\alpha} \cap (b]_{\alpha} = {\rm{D}}(A)$. Then ${\rm{D}}(A) \subseteq P$ and $(a]_{\alpha} \cap (b]_{\alpha} \subseteq P$.

On the other hand, by Remark \ref{cuasi-dense}, there is $d \in A$ such that $d^{\top} = \{1\}$ and $d \in (a]_{\alpha} \cap (b]_{\alpha}$. Then $d^{\top} \subseteq a^{\top}$ and $d^{\top} \subseteq b^{\top}$. We get 
\[
(a^{\top} \veebar d^{\top}) \cap (b^{\top} \veebar d^{\top}) =  \overline{m}(a^{\top}, b^{\top}, d^{\top}) = m(a,b,d)^{\top}.
\]  
It follows $m(a,b,d)^{\top} \subseteq a^{\top}$ and $m(a,b,d)^{\top} \subseteq b^{\top}$, i.e., $m(a,b,d) \in (a]_{\alpha} \cap (b]_{\alpha}$ and $m(a,b,d) \in P$. Then $(a \vee d) \wedge_{d} (b \vee d) \in P$. Since $a \notin P$, we have $b \vee d \in P$ and $b \in P$. Thus, $b \in P \cap a^{\top}$ and $P \cap a^{\top} \neq \emptyset$. Conversely, it is easy to see that if $P \cap a^{\top} \neq \emptyset$, then $a \notin P$. By Lemma \ref{prop_anni} we have $P \in {\rm{X}}_{m}(A)$.
\end{proof}

\begin{lemma} \label{P prime in cuasi}
Let ${\bf{A}}$ be a quasicomplemented distributive nearlattice. Let $P \in {\rm{X}}_{m}(A)$ and $a,b \in A$. If $(a^{\top} \cap b^{\top})^{\top} \cap P \neq \emptyset$, then $a \in P$ or $b \in P$.
\end{lemma}

\begin{proof}
Let $P \in {\rm{X}}_{m}(A)$ and $a,b \in A$. Since ${\bf{A}}$ is quasicomplemented, there exist $\tilde{a}, \tilde{b} \in A$ such that $a^{\top} = (\tilde{a})^{\top \top}$ and $b^{\top} = (\tilde{b})^{\top \top}$. So, by Lemma \ref{prop_anni},  
\[
(a^{\top} \cap b^{\top})^{\top} = ((\tilde{a})^{\top \top} \cap (\tilde{b})^{\top \top})^{\top} = (\tilde{a} \vee \tilde{b})^{\top}
\]
and $(\tilde{a} \vee \tilde{b})^{\top} \cap P \neq \emptyset$. As $P \in {\rm{X}}_{m}(A)$, by Lemma \ref{prop_anni} we have $\tilde{a} \vee \tilde{b} \notin P$. Thus, $\tilde{a} \notin P$ or $\tilde{b} \notin P$. If $\tilde{a} \notin P$, then $(\tilde{a})^{\top} \cap P \neq \emptyset$ and $a^{\top \top} \cap P \neq \emptyset$. It follows $a \in P$. Analogously, if $\tilde{b} \notin P$, then $b \in P$. Therefore, $a \in P$ or $b \in P$.
\end{proof}

We know that ${\rm{Fi}}_{\alpha}({\bf{A}})$ is a Heyting algebra where the join of two filters $F$ and $G$ of $A$ is defined as $F {\underline{\sqcup}} G = \alpha(F {\veebar} G)$. In the class of normal and quasicomplemented distributive nearlattices we have the following characterization of $F {\underline{\sqcup}} G$.

\begin{theorem} 
Let ${\bf{A}}$ be a normal and quasicomplemented distributive nearlattice. If $F,G \in {\rm{Fi}}_{\alpha}(A)$, then 
\[
F {\underline{\sqcup}} G = \{ x \in A \colon \exists (f,g) \in F \times G [ x \in (f^{\top} \cap g^{\top})^{\top} ] \}.
\]
\end{theorem}

\begin{proof}
Let $W = \{ x \in A \colon \exists (f,g) \in F \times G [ x \in (f^{\top} \cap g^{\top})^{\top} ] \}$. If $x \in F {\underline{\sqcup}} G = \alpha (F \veebar G)$, then there is $a \in F \veebar G = {\rm{Fig}}(F \cup G)$ such that $a^{\top} \subseteq x^{\top}$. Thus, there exist $f \in F$ and $g \in G$ such that $f \wedge g$ exists and $a = f \wedge g$. We have $a^{\top} = f^{\top} \cap g^{\top}$ and $x \in x^{\top \top} \subseteq a^{\top \top} = (f^{\top} \cap g^{\top})^{\top}$. Then $x \in W$. 

Conversely, if $z \in W$, then there exists $(f,g) \in F \times G$ such that $z \in (f^{\top} \cap g^{\top})^{\top}$. We suppose $z \notin F {\underline{\sqcup}} G$. Then by Theorem \ref{prime alpha_theo} there exists $P \in {\rm{X}}_{\alpha}(A)$ such that $z \in P$ and $P \cap \alpha (F \veebar G) = \emptyset$. Since ${\bf{A}}$ is quasicomplemented, by Theorem \ref{prime alpha implica maximal} we get $P \in {\rm{X}}_{m}(A)$. On the other hand, $z \in (f^{\top} \cap g^{\top})^{\top} \cap P$ and by Lemma \ref{P prime in cuasi}, $f \in P$ or $g \in P$. If $f \in P$, then $f \in \alpha (F \veebar G)$ and $P \cap \alpha (F \veebar G) \neq \emptyset$ which is a contradiction. Analogously if $g \in P$. We conclude that $z \in F {\underline{\sqcup}} G$.
\end{proof}

\begin{theorem}
Let ${\bf{A}}$ be a normal distributive nearlattice. If every dense $\alpha$-filter contains a dense element, then ${\bf{A}}$ is quasicomplemented.
\end{theorem}

\begin{proof}
Let $a \in A$. Consider the filter $F = a^{\top} \veebar a^{\top \top}$. Then $a^{\top} \subseteq \alpha(F)$ and $a^{\top \top} \subseteq \alpha(F)$. Since ${\rm{Fi}}({\bf{A}})$ is a pseudocomplemented distributive lattice, we have $\alpha(F)^{\top} \subseteq a^{\top}$ and $\alpha(F)^{\top} \subseteq a^{\top \top}$. So, $\alpha(F)^{\top} \subseteq a^{\top} \cap a^{\top \top} = \{1\}$ and $\alpha(F)^{\top} = \{1\}$. As $\alpha(F)$ is a dense $\alpha$-filter, by hypothesis, there is $x \in \alpha(F)$ such that $x^{\top} = \{1\}$. It follows there exists $f \in F$ such that $f^{\top} \subseteq x^{\top} = \{1\}$, i.e., $f^{\top} = \{1\}$. Thus, as $f \in F$, there exist $y \in a^{\top}$ and $z \in a^{\top \top}$ such that $y \wedge z$ exists and $f = y \wedge z$. Then $f^{\top} = (y \wedge z)^{\top} = y^{\top} \cap z^{\top} = \{1\}$. So, $y^{\top} \subseteq z^{\top \top}$ because $z^{\top \top}$ is the pseudocomplement of $z^{\top}$. Moreover, $z^{\top \top} \subseteq a^{\top \top}$ and $y^{\top} \subseteq a^{\top \top}$. On the other hand, $y \in a^{\top}$ implies $a^{\top \top} \subseteq y^{\top}$. Therefore, $a^{\top \top} = y^{\top}$ and ${\bf{A}}$ is quasicomplemented.
\end{proof}

\subsection{Connection between $\alpha$-ideals and $\alpha$-filters} 

Now we are going to study the relationship between $\alpha$-filters and $\alpha$-ideals in normal and quasicomplemented distributive nearlattices. Let ${\bf{A}}$ be a distributive nearlattice and $I \in {\rm{Id}}(A)$. By the result given in \cite{CaGo21} the set 
\[
F_{I} = \{ x \in A \colon \exists i \in I (i \in x^{\top}) \}
\]
is an $\alpha$-filter. In this way, we can naturally associate an $\alpha$-filter to each ideal of $A$. Conversely, let $F \in {\rm{Fi}}(A)$ and consider 
\[
I_{F} = \{ x \in A \colon \exists f \in F (x^{\top} \subseteq f^{\top \top}) \}. 
\]
The set $I_{F}$ will be relevant to the rest of this section.

\begin{lemma} \label{ideal I_F}
Let ${\bf{A}}$ be a normal and quasicomplemented distributive nearlattice. If $F$ is an $\alpha$-filter of $A$, then $I_{F}$ is an $\alpha$-ideal.
\end{lemma}

\begin{proof}
By Remark \ref{cuasi-dense} note that the set $I_{F}$ is non-empty. Let $x,y \in A$ be such that $x \leq y$ and $y \in I_{F}$. Then there is $f \in F$ such that $y^{\top} \subseteq f^{\top \top}$. By Lemma \ref{prop_anni}, $x^{\top} \subseteq y^{\top}$ and $x^{\top} \subseteq f^{\top \top}$, i.e., $x \in I_{F}$ and $I_{F}$ is decreasing. Let $x,y \in I_{F}$. Then there exist $f_{1}, f_{2} \in F$ such that $x^{\top} \subseteq f_{1}^{\top \top}$ and $y^{\top} \subseteq f_{2}^{\top \top}$. So, $f_{1}^{\top} \subseteq x^{\top \top}$, $f_{2}^{\top} \subseteq y^{\top \top}$ and by Lemma \ref{prop_anni}, $f_{1}^{\top} \cap f_{2}^{\top} \subseteq x^{\top \top} \cap y^{\top \top} = (x \vee y)^{\top \top}$. On the other hand, by Remark \ref{cuasi-dense}, there is $d \in A$ such that $d^{\top} = \{1\}$. Then $m(f_{1}, f_{2}, d) \in F$ and as ${\bf{A}}$ is normal,
\[
\begin{array}{rll}
m(f_{1}, f_{2}, d)^{\top} & = & \overline{m}(f_{1}^{\top}, f_{2}^{\top}, d^{\top}) \\
             & = & (f_{1}^{\top} \veebar \{1\}) \cap (f_{2}^{\top} \veebar \{1\}) \\
             & = & f_{1}^{\top} \cap f_{2}^{\top}. \\
\end{array}
\] 
Thus, $m(f_{1}, f_{2}, d)^{\top} \subseteq (x \vee y)^{\top \top}$. It follows $(x \vee y)^{\top} \subseteq m(f_{1}, f_{2}, d)^{\top \top}$ and $x \vee y \in I_{F}$. Hence, $I_{F}$ is an ideal of $A$. Prove that $I_{F}$ is an $\alpha$-ideal. Let $a \in A$ be such that $I_{F} \cap a^{\top \top} \neq \emptyset$. So, there is $x \in I_{F}$ such that $x \in a^{\top \top}$. Then there exists $f \in F$ such that $x^{\top} \subseteq f^{\top \top}$. By Lemma \ref{prop_anni}, $a^{\top} \subseteq x^{\top}$ and $a^{\top} \subseteq f^{\top \top}$, i.e., $a \in I_{F}$. 
\end{proof}

\begin{theorem} 
Let ${\bf{A}}$ be a normal and quasicomplemented distributive nearlattice. Then there exist a one-to-one correspondence between $\alpha$-ideals and $\alpha$-filters of $A$.
\end{theorem}

\begin{proof}
Let $I$ be an $\alpha$-ideal. We see that $I = I_{F_{I}}$. By Remark \ref{cuasi-dense}, $I$ is non-empty. Then $F_{I}$ is an $\alpha$-filter and by Lemma \ref{ideal I_F}, $I_{F_{I}}$ is an $\alpha$-ideal. Let $x \in I_{F_{I}}$. Then there exists $f \in F_{I}$ such that $x^{\top} \subseteq f^{\top \top}$. So, there is $i \in I$ such that $i \in f^{\top}$. It follows $f^{\top \top} \subseteq i^{\top}$ and $x^{\top} \subseteq i^{\top}$. Thus, $i^{\top \top} \subseteq x^{\top \top}$ and since $i \in i^{\top \top}$ we have $i \in I \cap x^{\top \top}$, i.e., $I \cap x^{\top \top} \neq \emptyset$. As $I$ is an $\alpha$-ideal, $x \in I$ and $I_{F_{I}} \subseteq I$. For the other inclusion, let $x \in I$ and consider the filter $x^{\top \top}$. Since ${\bf{A}}$ is quasicomplemented, there exists $y \in A$ such that $x^{\top \top} = y^{\top}$. Then $x \in y^{\top}$ and as $x \in I$  we have $y \in F_{I}$. On the other hand, $x^{\top} \subseteq y^{\top \top}$ and $y \in F_{I}$ implies $x \in I_{F_{I}}$. Hence, $I \subseteq I_{F_{I}}$.

Consider an $\alpha$-filter $F$ of $A$ and prove $F = F_{I_{F}}$. As $F$ is an $\alpha$-filter by Lemma \ref{ideal I_F} we know that $I_{F}$ is an $\alpha$-ideal. Since ${\bf{A}}$ is quasicomplemented, by Remark \ref{cuasi-dense} we have $I_{F}$ is a non-empty. Then $F_{I_{F}}$ is an $\alpha$-filter. Let $x \in F_{I_{F}}$. Then there exists $i \in I_{F}$ such that $i \in x^{\top}$. So, there is $f \in F$ such that $i^{\top} \subseteq f^{\top \top}$. It follows $x \in x^{\top \top} \subseteq i^{\top} \subseteq f^{\top \top}$, i.e., $x \in f^{\top \top}$ and since $F$ is an $\alpha$-filter, $f^{\top \top} \subseteq F$. Consequently $x \in F$ and $F_{I_{F}} \subseteq F$. Conversely, let $x \in F$ and we take the filter $x^{\top \top}$. As ${\bf{A}}$ is quasicomplemented, there exists $y \in A$ such that $x^{\top \top} = y^{\top}$. In particular, $y^{\top} \subseteq x^{\top \top}$ and $y \in I_{F}$. Moreover, since $x^{\top} = y^{\top \top}$, we have $y \in x^{\top}$. It follows $x \in F_{I_{F}}$ and therefore $F \subseteq F_{I_{F}}$. The result is followed.
\end{proof}

\subsection{Ideals-congruence-kernels} \label{Subsec3}

Let ${\bf{A}}$ be a distributive nearlattice and $\theta$ a congruence of a $A$. Then the equivalence class $|1|_{\theta} = {\rm{ker}} \theta = \{ a \in A \colon (a,1) \in \theta \}$ is called the {\it{kernel of $\theta$}}. So, $|1|_{\theta}$ is a filter of $A$. A subset $X$ of $A$ is called {\it{congruence-kernel}} if $X = {\rm{ker}} \theta$ for some congruence $\theta$ of $A$. If $I$ is an ideal of $A$, then it is easy to see that  
\[
\theta(I) = \{ (a,b) \in A \times A \colon \exists i \in I (a \vee i = b \vee i) \} 
\]
is a congruence of $A$. A congruence $\theta$ of $A$ is called {\it{ideal-congruence}} if $\theta = \theta(I)$ for some ideal $I$ of $A$. In particular, if $I$ is an $\alpha$-ideal, $\theta$ is called {\it{$\alpha$-ideal-congruence}}. The aim of this subsection is to characterize the ideal-congruence-kernels of a normal and quasicomplemented distributive nearlattice, and describe the smallest $\alpha$-ideal-congruence $\theta$ of $A$ such that $F = {\rm{ker}} \theta$ for some $\alpha$-filter $F$ of $A$.

\begin{theorem} \label{theo_con-ker-fil}
Let ${\bf{A}}$ be a normal and quasicomplemented distributive nearlattice. Let $F$ be a subset of $A$. Then $F$ is an $\alpha$-filter if and only if $F$ is an $\alpha$-ideal-congruence-kernel.
\end{theorem}

\begin{proof}
If $F$ is an $\alpha$-filter, then by Lemma \ref{ideal I_F} and Remark \ref{cuasi-dense} we have $I_{F}$ is an non-empty $\alpha$-ideal of $A$. We prove $I_{F} \cap a^{\top} \neq \emptyset$, for every $a \in F$. Suppose the contrary, i.e., there is $a \in F$ such that $I_{F} \cap a^{\top} = \emptyset$. By Lemma \ref{prop_anni} there exists $U \in {\rm{X}}_{m}(A)$ such that $I_{F} \subseteq U$ and $a \in U$. So, $U \cap a^{\top} = \emptyset$. On the other hand, as ${\bf{A}}$ is quasicomplemented, there exists $b \in A$ such that $a^{\top \top} = b^{\top}$. In particular, $b^{\top} \subseteq a^{\top \top}$ and since $a \in F$ we have $b \in I_{F}$. Thus, $b \in U$ and $b \in b^{\top \top} = a^{\top}$, i.e., $b \in U \cap a^{\top}$ which is impossible. Hence $I_{F} \cap a^{\top} \neq \emptyset$, for every $a \in F$.

Now we prove $F = {\rm{ker}} \theta(I_{F})$. If $a \in F$, then $I_{F} \cap a^{\top} \neq \emptyset$. So, there is $i \in I_{F}$ such that $i \in a^{\top}$. It is clear that $(a,1) \in \theta(I_{F})$ and $a \in {\rm{ker}} \theta(I_{F})$. For the other inclusion, let $a \in {\rm{ker}} \theta(I_{F})$. Then there is $i \in I_{F}$ such that $a \vee i = 1$. So, there exists $f \in F$ such that $i^{\top} \subseteq f^{\top \top}$ and $a \in f^{\top \top}$. Since $F$ is an $\alpha$-filter, $f^{\top \top} \subseteq F$ and $a \in F$. Consequently, $F = {\rm{ker}} \theta(I_{F})$ and $F$ is an $\alpha$-ideal-congruence-kernel.

Conversely, suppose that $F$ is an $\alpha$-ideal-congruence-kernel. Then there exists an $\alpha$-ideal $I$ of $A$ such that $F = {\rm{ker}} \theta(I)$. We see that $F$ is an $\alpha$-filter. It is clear that $1 \in F$ and $F$ is increasing. Let $x,y \in F$ such that $x \wedge y$ exists. Then there exist $i_{1}, i_{2} \in I$ such that $x \vee i_{1} = y \vee i_{2}$. Let $i = i_{1} \vee i_{2} \in I$. Thus, $x,y \in i^{\top}$ and since $i^{\top}$ is a filter, $x \wedge y \in i^{\top}$ and $x \wedge y \in {\rm{ker}} \theta(I)$. Then $F$ is a filter of $A$. Let $a \in F$ and $x \in a^{\top \top}$. As $a \in {\rm{ker}} \theta(I)$, there is $i \in I$ such that $a \vee i = 1$, i.e., $i \in a^{\top}$. It follows $a^{\top} \subseteq x^{\top}$ and $i \in x^{\top}$. Then $x \vee i = 1$ and $x \in {\rm{ker}} \theta(I) = F$. We conclude that $a^{\top \top} \subseteq F$ and $F$ is an $\alpha$-filter.
\end{proof}

In the proof of the Theorem \ref{theo_con-ker-fil} we have defined a congruence $\theta(I_{F})$ such that $F = {\rm{ker}} \theta(I_{F})$ for an $\alpha$-filter $F$. In the following result we determinate the smallest $\alpha$-ideal-congruence $\theta$ such that $F = {\rm{ker}} \theta$.

\begin{theorem} 
Let ${\bf{A}}$ be a normal and quasicomplemented distributive nearlattice. Let $F$ be an $\alpha$-filter of  $A$. Then $\theta(I_{F})$ is the smallest $\alpha$-ideal-congruence such that $F = {\rm{ker}} \theta(I_{F})$. 
\end{theorem}

\begin{proof}
By Lemma \ref{ideal I_F}, Remark \ref{cuasi-dense} and Theorem \ref{theo_con-ker-fil} we have that $I_{F}$ is an $\alpha$-ideal of $A$ such that $F = {\rm{ker}} \theta(I_{F})$. Let $J$ be an $\alpha$-ideal of $A$ such that $F = {\rm{ker}} \theta(J)$. If we prove $I_{F} \subseteq J$, then $\theta(I_{F}) \subseteq \theta(J)$. Let $x \in I_{F}$. Then there is $f \in F$ such that $x^{\top} \subseteq f^{\top \top}$. As $F = {\rm{ker}} \theta(J)$, there exists $j \in J$ such that $f \vee j = 1$. Then $f \in j^{\top}$ and $f^{\top \top} \subseteq j^{\top}$. It follows $x^{\top} \subseteq j^{\top}$ and by Lemma \ref{prop_anni}, $j^{\top \top} \subseteq x^{\top \top}$. Thus, $j \in J \cap x^{\top \top}$, i.e., $J \cap x^{\top \top} \neq \emptyset$ and since $J$ is an $\alpha$-ideal, $x \in J$ and $I_{F} \subseteq J$. 
\end{proof}

\section{Stone distributive nearlattices} \label{Sec4}

In this section we introduce the class of Stone distributive nearlattices and give different characterizations.  

\begin{definition}
Let ${\bf{A}}$ be a distributive nearlattice. We say that ${\bf{A}}$ is {\it{Stone}} if for each $a \in A$, we have $a^{\top} \veebar a^{\top \top} = A$.
\end{definition}

\begin{example}
The distributive nearlattice of Example \ref{normal no cuasi} is Stone.
\end{example}

\begin{example}
Recall that a join-semilattice with greatest element ${\bf{A}}$ is a {\it{semi-boolean algebra}} if each principal filter is a Boolean algebra \cite{Abb67}. Further, if ${\bf{A}}$ is a distributive nearlattice, then it was shown in \cite{Ca19} that ${\bf{A}}$ is a semi-boolean algebra if and only if $[a) \veebar a^{\top} = A$, for every $a \in A$. It follows that every semi-boolean algebra is a Stone distributive nearlattice.
\end{example}

\begin{lemma}
Let ${\bf{A}}$ be a Stone distributive nearlattice. Then $(a \wedge b)^{\top \top} = a^{\top \top} \veebar b^{\top \top}$, whenever $a \wedge b$ exists.
\end{lemma}

\begin{proof}
Let $a,b \in A$ be such that $a \wedge b$ exists. By Lemma \ref{prop_anni}, $a^{\top \top} \veebar b^{\top \top} \subseteq (a \wedge b)^{\top \top}$. We are going to show the converse inclusion. First note that if $F \in {\rm{Fi}}(A)$ such that $F^{\top} \cap (a \wedge b)^{\top} = \{1\}$, then $F^{\top} \cap a^{\top} \subseteq b^{\top \top}$. Indeed, by Lemma \ref{prop_anni} we have $F^{\top} \cap a^{\top} \cap b^{\top} = \{1\}$ and since $b^{\top \top}$ is the pseudocomplement of $b^{\top}$ in ${\rm{Fi}}({\bf{A}})$, we have $F^{\top} \cap a^{\top} \subseteq b^{\top \top}$. Let $F \in {\rm{Fi}}(A)$ be such that $F^{\top} \cap (a \wedge b)^{\top} = \{1\}$. Prove $F^{\top} \subseteq a^{\top \top} \veebar b^{\top \top}$. Since ${\bf{A}}$ is Stone and ${\rm{Fi}}({\bf{A}})$ is a distributive lattice, 
\[
\begin{array}{rll}
F^{\top} & = & F^{\top} \cap A \\
             & = & F^{\top} \cap (a^{\top} \veebar a^{\top \top}) \\
             & = & (F^{\top} \cap a^{\top}) \veebar (F^{\top} \cap a^{\top \top}) \\
             & \subseteq & b^{\top \top} \veebar a^{\top \top}. \\
\end{array}
\] 
In particular, if we take $F = (a \wedge b)^{\top} \in {\rm{Fi}}(A)$, then $(a \wedge b)^{\top \top} \subseteq b^{\top \top} \veebar a^{\top \top}$.
\end{proof}

\begin{theorem} \label{Stone_complemento}
Let ${\bf{A}}$ be a distributive nearlattice. Then the following conditions are equivalent:
\begin{enumerate}
\item ${\bf{A}}$ is Stone.
\item Every element of ${\rm{R}}(A)$ has its complement in ${\rm{Fi}}(A)$.
\end{enumerate}
\end{theorem}

\begin{proof}
$(1) \Rightarrow (2)$ Let $a^{\top} \in {\rm{R}}(A)$. Then by Lemma \ref{prop_anni} we have $a^{\top} \cap a^{\top \top} = \{1\}$, and since ${\bf{A}}$ is Stone, $a^{\top} \veebar a^{\top \top} = A$. Thus, $a^{\top}$ is complemented in ${\rm{Fi}}(A)$.

$(2) \Rightarrow (1)$ If $a \in A$, then $a^{\top} \in {\rm{R}}(A)$ and by hypothesis $a^{\top}$ is complemented in ${\rm{Fi}}(A)$, i.e., there exists $F \in {\rm{Fi}}(A)$ such that $a^{\top} \cap F = \{1\}$ and $a^{\top} \veebar F = A$. Thus, as $a^{\top \top}$ is the pseudocomplement of $a^{\top}$ in ${\rm{Fi}}(A)$, we get $F \subseteq a^{\top \top}$ and $A = a^{\top} \veebar F \subseteq a^{\top} \veebar a^{\top \top}$, i.e., $a^{\top} \veebar a^{\top \top} = A$. Hence, ${\bf{A}}$ is Stone.
\end{proof}

\begin{corollary}
Let ${\bf{A}}$ be a distributive nearlattice. If ${\bf{A}}$ is normal and quasicomplemented, then ${\bf{A}}$ is Stone.
\end{corollary}
\begin{proof}
It follows by Corollary \ref{normal-quasi-R(A)} and Theorem \ref{Stone_complemento}.
\end{proof}

\begin{remark}
In the Example \ref{cuasi no normal}, the annihilators $a^{\top}$ and $b^{\top}$ have no complement in ${\rm{Fi}}(A)$. Then by Theorem \ref{Stone_complemento} we have ${\bf{A}}$ is not Stone.
\end{remark}

\begin{theorem} \label{Stone implica normal}
Let ${\bf{A}}$ be a distributive nearlattice. If ${\bf{A}}$ is Stone, then ${\bf{A}}$ is normal.
\end{theorem}

\begin{proof}
Let $a,b \in A$. By Lemma \ref{prop_anni} we have $a^{\top} \veebar b^{\top} \subseteq (a \vee b)^{\top}$. Prove the other inclusion. Let $x \in (a \vee b)^{\top}$. Then $x \vee b \in a^{\top}$ and 
\[
[x \vee b) \cap a^{\top \top} = [x) \cap [b) \cap a^{\top \top} \subseteq a^{\top} \cap a^{\top \top} = \{1\}. 
\]
Thus, $[x) \cap a^{\top \top} \subseteq [b)^{\top} = {\rm{Fig}}(\{b\})^{\top} = b^{\top}$. Indeed, if $y \in [x) \cap a^{\top \top}$ we have $y \vee b \in [x) \cap [b) \cap a^{\top \top} = \{1\}$ and $y \vee b = 1$, i.e., $y \in b^{\top}$. Then, as ${\bf{A}}$ is Stone and ${\rm{Fi}}({\bf{A}})$ is a distributive lattice,  
\[
\begin{array}{rll}
x \in [x) & = & [x) \cap A \\
            & = & [x) \cap \left( a^{\top} \veebar a^{\top \top} \right) \\
            & = & \left( [x) \cap a^{\top} \right) \veebar \left( [x) \cap a^{\top \top} \right) \\
            & \subseteq & a^{\top} \veebar b^{\top}. \\
\end{array}
\]
So, $x \in a^{\top} \veebar b^{\top}$ and $(a \vee b)^{\top} = a^{\top} \veebar b^{\top}$. By Theorem \ref{normal}, ${\bf{A}}$ is normal. 
\end{proof}

Now we can to prove the converse of Corollary \ref{normal-quasi-R(A)}.

\begin{theorem}
Let ${\bf{A}}$ be a distributive nearlattice. Then the following conditions are equivalent:
\begin{enumerate}
\item ${\bf{A}}$ is normal and quasicomplemented.
\item ${\rm{R}}({\bf{A}})$ is a Boolean algebra.
\end{enumerate}
\end{theorem}

\begin{proof}
$(1) \Rightarrow (2)$ It follows from Corollary \ref{normal-quasi-R(A)}.

$(2) \Rightarrow (1)$ By Theorem \ref{Stone_complemento} and \ref{Stone implica normal}, ${\bf{A}}$ is normal. Prove that ${\bf{A}}$ is quasicomplemented. Let $a \in A$. Then $a^{\top} \in {\rm{R}}(A)$ and by hypothesis there exists $b^{\top} \in {\rm{R}}(A)$ such that $a^{\top} \cap b^{\top} = \{1\}$ and $a^{\top} \veebar b^{\top} = A$. Since $a^{\top} \cap b^{\top} = \{1\}$ and $a^{\top \top}$ is the pseudocomplement of $a^{\top}$ in ${\rm{Fi}}(A)$ we have $b^{\top} \subseteq a^{\top \top}$. We see the other inclusion. Let $z \in a^{\top \top}$. As $b \in A = a^{\top} \veebar b^{\top} = {\rm{Fig}}(a^{\top} \cup b^{\top})$ then there exist $x \in a^{\top}$ and $y \in b^{\top}$ such that $x \wedge y$ exists and $b = x \wedge y$. So, $z \vee x = 1$. On the other hand, 
\[
z \vee b = z \vee (x \wedge y) = (z \vee x) \wedge (z \vee y) = 1 \wedge (z \vee y) = z \vee y.
\]
As $y \in b^{\top}$, then  
\[
z \vee b = (z \vee y) \vee b = z \vee (y \vee b) = 1
\]
and $z \in b^{\top}$. So, $a^{\top \top} \subseteq b^{\top}$ and the proof is complete.
\end{proof}

\begin{theorem}
Let ${\bf{A}}$ be a quasicomplemented distributive nearlattice. Then the following conditions are equivalent:
\begin{enumerate}
\item ${\bf{A}}$ is Stone.
\item Every prime ideal is contained in a unique prime $\alpha$-ideal.
\end{enumerate}
\end{theorem}

\begin{proof}
$(1) \Rightarrow (2)$ Let $P \in {\rm{X}}(A)$. By Theorem \ref{Stone implica normal}, ${\bf{A}}$ is normal and by Definition \ref{def_normal} there is a unique $U \in {\rm{X}}_{m}(A)$ such that $P \subseteq U$. Then by Lemma \ref{prop_anni}, $U \in {\rm{X}}_{\alpha}(A)$. As ${\bf{A}}$ is normal, it follows by Theorem \ref{prime alpha implica maximal} that $U$ is unique.  

$(2) \Rightarrow (1)$ Suppose there is $a \in A$ such that $a^{\top} \veebar a^{\top \top} \subset A$, i.e., there is $x \in A$ such that $x \notin a^{\top} \veebar a^{\top \top}$. By Theorem \ref{prime_theo} there exists $P \in {\rm{X}}(A)$ such that $x \in P$ and $P \cap (a^{\top} \veebar a^{\top \top}) = \emptyset$. So, $P \cap a^{\top} = \emptyset$ and $P \cap a^{\top \top} = \emptyset$. As ${\bf{A}}$ is quasicomplemented, there exists $b \in A$ such that $a^{\top \top} = b^{\top}$. By Lemma \ref{prop_anni} there exist $U_{1}, U_{2} \in {\rm{X}}_{m}(A)$ such that $P \subseteq U_{1}$, $P \subseteq U_{2}$, $a \in U_{1}$ and $b \in U_{2}$. So, $U_{1}, U_{2} \in {\rm{X}}_{\alpha}(A)$ and by hypothesis $U_{1}=U_{2}$. Therefore $a,b \in U_{1}$ and as $a \in b^{\top}$ we get $a \vee b = 1 \in U_{1}$, which is a contradiction. Then ${\bf{A}}$ is Stone.  
\end{proof}

\subsection{$\sigma$-filters}

Our next aim is to characterize the Stone distributive nearlattices through certain particular filters. Let ${\bf{A}}$ be a distributive nearlattice. Let $F$ be a filter of $A$ and consider the set 
\[
\sigma(F) = \{ x \in A \colon x^{\top} \veebar F = A \}.
\]

\begin{lemma} \label{lem_sigma}
Let ${\bf{A}}$ be a distributive nearlattice. If $F$ is a filter of $A$, then $\sigma(F)$ is an $\alpha$-filter such that $\sigma(F) \subseteq F$.
\end{lemma}

\begin{proof}
As $1^{\top}=A$ it follows $1 \in \sigma(F)$. Let $x,y \in A$ such that $x \leq y$ and $x \in \sigma(F)$. Then by Lemma \ref{prop_anni} we have $x^{\top} \subseteq y^{\top}$ and $x^{\top} \veebar F = A$. So, $x^{\top} \veebar F \subseteq y^{\top} \veebar F$ and $y^{\top} \veebar F = A$, i.e., $y \in \sigma(A)$. Let $x,y \in \sigma(F)$ be such that $x \wedge y$ exists. Thus, $x^{\top} \veebar F = A$ and $y^{\top} \veebar F = A$. Since ${\rm{Fi}}({\bf{A}})$ is a distributive lattice, 
\[
(x \wedge y)^{\top} \veebar F = (x^{\top} \cap y^{\top}) \veebar F = (x^{\top} \veebar F) \cap (y^{\top} \veebar F) = A.
\]
Then $x \wedge y \in \sigma(F)$ and $\sigma(F)$ is a filter. We prove that $\sigma(F)$ is an $\alpha$-filter. Let $a \in \sigma(F)$ and $x \in a^{\top \top}$. Then $a^{\top} \subseteq x^{\top}$ and $a^{\top} \veebar F \subseteq x^{\top} \veebar F$. Since $a^{\top} \veebar F = A$ then $x \in \sigma(F)$ and $a^{\top \top} \subseteq \sigma(F)$. Finally, we see $\sigma(F) \subseteq F$. If $x \in \sigma(F)$, then $x^{\top} \veebar F = A$. So, $x \in x^{\top} \veebar F$ and there exist $y \in x^{\top}$ and $f \in F$ such that $y \wedge f$ exists and $x = y \wedge f$. Then $x \leq f$ and 
\[
f = f \wedge 1 = f \wedge (y \vee x) = (f \wedge y) \vee (f \wedge x) = x. 
\]
Hence $x \in F$ and $\sigma(F) \subseteq F$.
\end{proof}

\begin{lemma} \label{Stone_sigma}
Let ${\bf{A}}$ be a Stone distributive nearlattice. If $F$ is a filter of $A$, then $\sigma(F)$ is the greatest $\alpha$-filter such that $\sigma(F) \subseteq F$. 
\end{lemma}

\begin{proof}
By Lemma \ref{lem_sigma} $\sigma(F)$ is an $\alpha$-filter of $A$ such that $\sigma(F) \subseteq F$. Let $G \in {\rm{Fi}}_{\alpha}(A)$ be such that $G \subseteq F$ and $x \in G$. Since $G$ is an $\alpha$-filter we have $x^{\top \top} \subseteq G$. Then $x^{\top} \veebar x^{\top \top} \subseteq x^{\top} \veebar G \subseteq x^{\top} \veebar F$, and as ${\bf{A}}$ is Stone, $x^{\top} \veebar F = A$ and $x \in \sigma(F)$. 
\end{proof}

\begin{definition}
Let ${\bf{A}}$ be a distributive nearlattice. We say that a filter $F$ of $A$ is a {\it{$\sigma$-filter}} if $\sigma(F) = F$.
\end{definition}

We finish this paper with a characterization of Stone distributive nearlattices.

\begin{theorem} 
Let ${\bf{A}}$ be a distributive nearlattice. Then ${\bf{A}}$ is Stone if and only if every $\alpha$-filter is a $\sigma$-filter.
\end{theorem}
\begin{proof}
Let $F$ be an $\alpha$-filter of $A$. By Lemma \ref{lem_sigma}, $\sigma(F) \subseteq F$. Since ${\bf{A}}$ is Stone and $F$ is an $\alpha$-filter, by Lemma \ref{Stone_sigma} if follows $F \subseteq \sigma(F)$. So, $F$ is a $\sigma$-filter.

Conversely, let $a \in A$. Then every $\alpha$-filter is a $\sigma$-filter. In particular, the filter $a^{\top \top}$ is an $\alpha$-filter and $a \in a^{\top \top} = \sigma(a^{\top \top})$. So, $a^{\top} \veebar a^{\top \top} = A$.
\end{proof}

\end{document}